\documentclass[a4paper,10pt]{amsart}
\usepackage{amsmath, amssymb}
\usepackage{amscd}
\usepackage{color}
\usepackage{comment}

\usepackage[colorlinks]{hyperref}
\definecolor{linkblue}{RGB}{1,1,190}
\definecolor{citered}{RGB}{190,1,1}
\hypersetup{
	linkcolor=linkblue,
	urlcolor=linkblue,
	citecolor=citered
}

\def\arxiv#1{{\small\href{http://www.arxiv.org/abs/#1}{\path{arXiv:#1}}}}

\theoremstyle{plain}
\newtheorem{theorem}{\bf Theorem}[section]

\newtheorem{corollary}[theorem]{\bf Corollary}

\theoremstyle{definition}

\newcommand\ec { \color{black}}

\newcommand{\op}{{\text{\rm op}}}

\newcommand{\N}{\mathbb N}

\newcommand{\bdot}{\boldsymbol{\cdot}}

 \DeclareMathOperator{\ord}{ord}

\numberwithin{equation}{section}

\begin{document}

\title[On the isomorphism problem for monoids of product-one sequences]{On the isomorphism problem for \\ monoids of product-one sequences}

\author{Alfred Geroldinger and Jun Seok Oh}

\address{Department of Mathematics and Scientific Computing, University of Graz, NAWI Graz, Heinrichstraße 36, 8010 Graz, Austria}
\email{alfred.geroldinger@uni-graz.at, https://imsc.uni-graz.at/geroldinger}

\address{Department of Mathematics Education, Jeju National University, Jeju 63243, Republic of Korea}
\email{junseok.oh@jejunu.ac.kr}
\urladdr{}

\thanks{This work was supported by the Austrian Science Fund FWF, Project Number P36852, and by the National Research Foundation of Korea (NRF) grant funded by the Korea government (MSIT) (NRF-2021R1G1A1095553).}

\subjclass{11B30, 20M13, 20M14}
\keywords{Isomorphism problems, torsion groups, product-one sequences}

\begin{abstract}
Let $G_1$ and $G_2$ be torsion groups. We prove that the monoids of product-one sequences over $G_1$ and over $G_2$ are isomorphic if and only if the groups $G_1$ and $G_2$ are isomorphic. This was known before for abelian groups.
\end{abstract}

\maketitle

\smallskip
\section{Introduction} \label{1}
\smallskip

Let $G$ be a group. Elements of the free abelian monoid over $G$ are called sequences over $G$ whence -- in combinatorial terms -- sequences are finite unordered sequences with repetition allowed. A sequence is a product-one sequence if its terms can be ordered so that their product is the identity element of $G$. The set $\mathcal B (G)$ of all product-one sequences is a submonoid of the free abelian monoid over $G$ (in combinatorial terms, the underlying operation is just the concatenation of sequences) and it is called the monoid of product-one sequences over $G$.

In case of abelian groups, additive notation and terminology are dominant, whence the term zero-sum sequences is used.
Pushed forward by a huge variety of connections in number theory, algebra, discrete geometry, and combinatorics, the combinatorial and algebraic properties of zero-sum sequences and their associated monoids have been under intense study for decades. Classic combinatorial zero-sum invariants include the Davenport constant and the {E}rd{\H{o}}s-{G}inzburg-{Z}iv constant. For strong recent progress on the latter see \cite{Sa-Za24a}, and for  some books and surveys whose focus is on algebraic applications, see \cite{Ge-HK06a, Ge-Ru09, Gr13a, Sc16a}.

The investigation of combinatorial invariants of product-one sequences over non-abelian groups goes also back to the 1960s. For recent progress, we refer to \cite{Gr13b, Ge-Gr13a, Br-Ri18a, Oh-Zh20b, Ga-Li-Qu21, Zh21, Qu-Li-Tee22a, Qu-Li22a, Qu-Li22b, Av-Br-Ri23a, An-Cz-Do-Sz25}.
Since the last decade algebraic properties of the monoid of  product-one sequences have found  increased interest, which was pushed forward from new applications both in invariant theory (in particular, in the study of the Noether Number) and in the factorization theory of monoids and domains (e.g., monoids of product-one sequences over finite groups are C-monoids, which are a combinatorial analogue of the more complicated C-domains, \cite{Re13a, Ge-Zh19c, Oh22a}). For a sample of papers in these directions, see \cite{Cz-Do-Ge16a, Cz-Do-Sz18, Fa-Zh22a, Ha-Zh19a, Oh19b, Oh20a, G-G-O-Z22a}.

In the present paper, we study the Isomorphism Problem for monoids of product-one sequences, which runs as follows.

\smallskip
\noindent
{\bf The Isomorphism Problem.}
Let $G_1$ and $G_2$ be    groups  such that the monoids $\mathcal B (G_1)$ and  $\mathcal B (G_2)$ are isomorphic. Are the groups $G_1$ and $G_2$  isomorphic?

\smallskip
The Isomorphism Problem has an affirmative answer for abelian groups, and the argument runs as follows. Suppose that $G_1$ and $G_2$ are abelian and, to exclude a trivial case, suppose that both have at least three elements. Then the monoids $\mathcal B (G_1)$ and $\mathcal B (G_2)$ are Krull and their class groups are isomorphic to $G_1$ and $G_2$. Since isomorphic Krull monoids have isomorphic class groups, we are done (for details see \cite[Corollary 2.5.7]{Ge-HK06a}). An affirmative answer to the Isomorphism Problem is a necessary condition for an affirmative answer to the Characterization Problem.

\smallskip
\noindent
{\bf The Characterization Problem.}
Let $G_1$ and $G_2$ be  finite  groups with Davenport constant $\mathsf D (G_1) \ge 4$ such that their systems of sets of lengths $\mathcal L (G_1)$ and $\mathcal L (G_2)$ coincide. Are the groups $G_1$ and $G_2$  isomorphic?

\smallskip
As usual,
\[
\mathcal L (G) = \big\{ \mathsf L (B) \colon B \in \mathcal B (G) \big\}
\]
denotes the collection of sets of  lengths $\mathsf L (B)$, where $\mathsf L (B)$ is the set of all factorization lengths  $k \in \N$ for which $B = U_1 \bdot \ldots \bdot U_k$ for some minimal product-one sequences $U_1, \ldots, U_k$.
The standing conjecture is that the Characterization Problem has an affirmative answer for finite abelian groups (for this topic, see \cite{Pl-Sc19a},  the survey \cite{Ge-Zh20a}, and note that $\mathcal L (G_1)=\mathcal L (G_2)$ for any two infinite abelian groups $G_1$ and $G_2$). In the non-abelian setting, a first step was done  for finite groups with  Davenport constant 6 (\cite[Theorem 4.7]{Oh19b}) and for finite dihedral groups (\cite[Corollary 6.13]{G-G-O-Z22a}).

Let us go back to the Isomorphism Problem in the non-abelian setting. The argument, which we sketched above for abelian groups, does not carry over to the general case, because for a given group $G$, its monoid of product-one sequences $\mathcal B (G)$ is Krull if and only if $G$ is abelian (\cite[Theorem 3.14]{Fa-Zh22a}). Nevertheless, for finite groups, an analogous strategy could run as follows. Let $G_1$ and $G_2$ be finite groups. Then $\mathcal B (G_1)$ and $\mathcal B (G_2)$ are C-monoids (\cite[Theorem 3.2]{Cz-Do-Ge16a}). If the monoids are isomorphic, then their class semigroups (which coincide with the usual class groups in the abelian case) are isomorphic. However, so far we know too little about the structure of the class semigroups in order to get  back information on the groups. Thus, we choose a different strategy. We build substantially on recent results by Fadinger and Zhong (\cite{Fa-Zh22a}), which allow us to   give an affirmative answer to the Isomorphism Problem for torsion groups.

\smallskip
\begin{theorem} \label{1.1}
Let $G_1$ and $G_2$ be groups and suppose that $G_1$ is a torsion group. Then the groups $G_1$ and $G_2$ are isomorphic if and only if their monoids of product-one sequences $\mathcal B (G_1)$ and $\mathcal B (G_2)$ are isomorphic.
\end{theorem}

The general case of non-torsion groups remains open, but we obtain the following corollary.

\smallskip
\begin{corollary} \label{1.2}
Let $G_1$ and $G_2$ be groups whose commutator subgroups are torsion.
If the monoids $\mathcal B (G_1)$ and $\mathcal B (G_2)$ are isomorphic, then the torsion subgroups $\mathsf T (G_1)$ and $\mathsf T (G_2)$ are isomorphic.
\end{corollary}

In Section \ref{2}, we gather the required background. The proof of Theorem \ref{1.1} will be given in Section \ref{3} and the proof of Corollary \ref{1.2} will be given in Section \ref{4}.

\smallskip
\section{Prerequisites} \label{2}
\smallskip

By a {\it monoid}, we mean a commutative cancellative semigroup with identity element. Let $H$ be a multiplicatively written monoid. Then $H^{\times}$ denotes its group of invertible elements, $\mathsf q (H)$ its quotient group,
\begin{itemize}
\item $\widetilde H = \{ x \in \mathsf q (H) \colon \text{there is $n \in \N$ such that $x^n \in H$} \}$  the {\it root closure} of $H$, and

\item $\widehat H = \{ x \in \mathsf q (H) \colon \text{there is $c \in H$ such that $cx^n \in H$ for all $n \in \N$} \}$  the {\it complete integral closure} of $H$.
\end{itemize}
Then, we have $H \subset \widetilde H \subset \widehat H \subset \mathsf q (H)$, and $H$ is {\it root closed} (resp., {\it completely integrally closed}) if $H = \widetilde H$ (resp., $H = \widehat H$). The monoid $H$ is {\it Krull} if it is completely integrally closed and satisfies the ascending chain condition on divisorial ideals. A {\it divisor theory} of $H$ is a homomorphism $\varphi \colon H \to D$ to a free abelian monoid $D$ which satisfies the following two properties.
\begin{itemize}
\item For all $a, b \in H$, we have $a$ divides $b$ (in $H$) if and only if $\varphi (a)$ divides $\varphi (b)$ (in $D$).

\item For all $\alpha \in D$, there are $a_1, \ldots, a_m \in H$ such that $\alpha = \gcd \big( \varphi (a_1), \ldots, \varphi (a_m) \big)$.
\end{itemize}
Then $H$ has a divisor theory if and only if it is Krull (\cite[Theorem 2.4.8]{Ge-HK06a}).
An element $a \in H$  is said to be {\it irreducible} (or an {\it atom}) if $a \notin H^{\times}$ and $a=bc$ with $b, c \in H$ implies that $b \in H^{\times}$ or $c \in H^{\times}$.

Let $G$ be a multiplicatively written group  with identity element $1_G=1$, and let $G_0 \subset G$ be a subset. We denote by $\mathcal F (G_0)$ the (multiplicatively written) free abelian monoid with basis $G_0$. The elements of $\mathcal F (G_0)$ are called {\it sequences} over $G_0$. Thus, in combinatorial terms, a sequence means a finite unordered sequence of terms from $G_0$  with the repetition of elements allowed. We have to distinguish carefully between the group operation in $G$ and the monoid operation in $\mathcal F (G_0)$ (in this regard, our notation coincides with the notation in all recent articles, including \cite{Gr13b, Fa-Zh22a}). To do so, we use the bold symbol $\bdot$ for the multiplication in $\mathcal F (G_0)$, whence $G = (G, \cdot)$ and $\mathcal F (G_0) = ( \mathcal F (G_0), \bdot)$. In order to avoid confusion between exponentiation in $G$ and exponentiation in $\mathcal F (G_0)$, we use brackets to denote exponentiation in $\mathcal F (G_0)$. Thus, for $g \in G_0$, $S \in \mathcal F (G_0)$, and $k \in \N_0$, we have
\[
  g^k \in G \,, \quad g^{[k]} = \underbrace{g \bdot \ldots \bdot g}_{k} \in \mathcal F (G_0) \,, \quad \text{and} \quad S^{[k]} = \underbrace{S \bdot \ldots \bdot S}_{k} \in \mathcal F (G_0) \,.
\]
Let
\[
  S = g_1 \bdot \ldots \bdot g_{\ell} = {\small \prod}^{\bullet}_{g \in G_0} g^{[\mathsf v_g (S)]}
\]
be a sequence over $G_0$. Then
\begin{itemize}
\item $\pi (S) = \{ g_{\tau (1)} \cdot \ldots \cdot g_{\tau (\ell)} \in G \colon \tau \ \text{is a permutation of } \ [1, \ell] \} \subset G$ is the {\it set of products} of $S$.
\end{itemize}
The sequence $S$ is called
\begin{itemize}
\item a {\it product-one sequence} if $1_G \in \pi (S)$,
\end{itemize}
and any ordered product in $\pi (S)$, that equals $1_G$, is called a {\it product-one equation} of $S$.
Then the set
\[
\mathcal B (G_0) = \{ S \in \mathcal F (G_0) \colon 1_G \in \pi (S) \} \subset \mathcal F (G_0)
\]
is a submonoid of $\mathcal F (G_0)$, called the {\it monoid of product-one sequences over $G_0$}. We denote by $\mathcal A (G_0)$ the set of atoms of $\mathcal B (G_0)$. It is easy to see that every $B \in \mathcal B (G_0)$ has a factorization into atoms and that there are only finitely many distinct factorizations. For any subset $G^{*}_0$ of a group $G^{*}$ and any map $\varphi \colon G_0 \to G^{*}_0$, we obtain a monoid homomorphism $\varphi \colon \mathcal F (G_0) \to \mathcal F (G^{*}_0)$, defined by $\varphi (S) = \varphi (g_1) \bdot \ldots \bdot \varphi (g_{\ell})$. We frequently use the following simple observation. Let $\ell \in \N$ and let $g_1, \ldots, g_{\ell} \in G$ be such that their product $g_1 \cdot \ldots \cdot g_{\ell} = 1_G$.
Then $g_{\ell} g_1 \cdot \ldots \cdot g_{\ell - 1} = g_{\ell}(g_1 \cdot \ldots \cdot g_{\ell})g_{\ell}^{-1} = g_{\ell}g_{\ell}^{-1} = 1_G$. Iterating this argument shows that
\[
  g_j \cdot \ldots \cdot g_{\ell} g_1 \cdot \ldots \cdot g_{j-1} = 1_G \quad \text{for all $j \in [1, \ell]$} \,.
\]

\smallskip
The  {\it opposite group} $G^{\op} = (G, \cdot^{\op})$ of $G=(G, \cdot)$ has the same underlying set and its group operation is defined by $g_1 \cdot^{\op} g_2 := g_2 \cdot g_1 = g_2g_1$ for all $g_1, g_2 \in G$. The map $\psi \colon G \to G^{\op}$, defined by $\psi (g)=g^{-1}$ for all $g \in G$, is a group isomorphism. Let $G_1$ and $G_2$ be groups and let $\varphi \colon G_1 \to G_2$ be a map. Then $\varphi^{\op} \colon G_1 \to G_2^{\op}$ is defined by $\varphi^{\op} (g) = \varphi ( {\ec g} )$ for all $g \in G_1$, and $\varphi^{\op}$ is a homomorphism if and only if $\varphi (g_1g_2) = \varphi (g_2) \varphi (g_1)$ for all $g_1, g_2 \in G_1$. In this case, $\varphi$ is called an {\it anti-homomorphism}.

\smallskip
\section{Proof of  Theorem \ref{1.1}} \label{3}
\smallskip

Before we start with the actual proof, we would like to mention that a result, weaker than Theorem \ref{1.1}, is already known. Indeed, for $i \in [1,2]$, let $G_i$ be a torsion group with $|G_i|>2$. Then the class group of the Krull monoid $\widehat{\mathcal B (G_i)}$ is isomorphic to $G_i/G_i'$, where $G_i'$ is the commutator subgroup of $G_i$ (\cite[Proposition 3.3 and Lemma 3.4]{Fa-Zh22a}). Thus, if the monoids of product-one sequences $\mathcal B (G_1)$ and $\mathcal B (G_2)$ are isomorphic, then the factor groups
\[
  G_1/G_1' \quad \mbox{ and } \quad G_2/G_2'
\]
are isomorphic. We will show that even $G_1$ and $G_2$ are isomorphic.

\smallskip
\begin{proof}[Proof of Theorem \ref{1.1}]
Let $G_1$ and $G_2$ be groups and suppose that $G_1$ is a torsion group. Clearly, if the groups are isomorphic, then their monoids of product-one sequences are isomorphic.

Conversely,  suppose that the monoids $\mathcal B (G_1)$ and $\mathcal B (G_2)$ are isomorphic.  By \cite[Theorem 3.14]{Fa-Zh22a}, a group $G$ is abelian if and only if $\mathcal B (G)$ is Krull if and only if $\mathcal B (G)$ is completely integrally closed. Thus, if $G_1$ or $G_2$ is abelian, then $\mathcal B (G_1)$ and $\mathcal B (G_2)$ are both Krull, and both groups are abelian. Thus, the claim follows from \cite[Corollary 2.5.7]{Ge-HK06a}. Now, we may suppose that neither $G_1$ nor $G_2$ is abelian. In particular, we have $|G_1| \ge 6$ and $|G_2|\ge 6$.

Let $i \in [1,2]$. Since the free abelian monoid $\mathcal F (G_i)$ is  completely integrally closed and since $\mathcal B (G_i) \subset \mathcal F (G_i)$, it follows that $\widehat{\mathcal B (G_i)} \subset \mathcal F (G_i)$. Since $G_1$ is a  torsion group, we have $\widetilde{\mathcal B (G_1)} = \widehat{\mathcal B (G_1)}$ by \cite[Lemma 3.4]{Fa-Zh22a}.

Since $\mathcal B (G_1)$ and $\mathcal B (G_2)$ are isomorphic, we obtain an isomorphism between the quotient groups $\mathsf q \big( \mathcal B (G_1) \big)$ and $\mathsf q \big( \mathcal B (G_2) \big)$, and an isomorphism between their complete integral closures, say $\phi \colon \widehat{\mathcal B (G_1)} \to \widehat{\mathcal B (G_2)}$.
By \cite[Proposition 3.3 and Lemma 3.4]{Fa-Zh22a}, the inclusion $\widehat{\mathcal B (G_i)} \hookrightarrow \mathcal F (G_i)$ is a divisor theory (see Section \ref{2}).
By the Uniqueness Theorem for divisor theories (\cite[Theorem 2.4.7]{Ge-HK06a}), there exists an isomorphism $\varphi \colon \mathcal F (G_1) \to \mathcal F (G_2)$ satisfying $\varphi |_{\widehat{\mathcal B (G_1)}} = \phi$. Every monoid isomorphism between free abelian monoids stems from a bijection between the basis sets.
Thus, we have a bijection between the groups, which we again denote by $\varphi$. Thus,  it remains to prove that $\varphi  \colon G_1 \to G_2$ is either a group homomorphism or a group anti-homomorphism. This will be done in a series of eight assertions.

\begin{enumerate}
\item[{\bf A1.}]  For all $g \in G_1$, we have $\ord(g) = \ord \big(\varphi(g)\big)$.
                        In particular, $G_2$ is a torsion group and $\varphi (1_{G_1}) = 1_{G_2}$.
                        Moreover, for every $S = g_1 \bdot \ldots \bdot g_{\ell} \in \mathcal F (G_1)$, we have $\varphi (S) = \varphi (g_1) \bdot \ldots \bdot \varphi (g_{\ell})$, and
                        \[
                          S \in \mathcal B (G_1) \quad \mbox{ if and only if } \quad \varphi (S) \in \mathcal B (G_2) \,.
                         \]

\item[{\bf A2.}]  For all $g \in G_1$, we have $\varphi (g^{-1} ) = \varphi (g)^{-1}$.

\item[{\bf A3.}]  For all $g_1, g_2 \in G_1$, we have
                  \[
                    \varphi (g_1 g_2) = \varphi (g_1) \varphi (g_2) \quad \mbox{ or } \quad \varphi (g_1 g_2) = \varphi (g_2) \varphi (g_1) \,.
                  \]
                  In particular, we have $\varphi (g^{n}) = \varphi (g)^{n}$ for every $n \in \mathbb Z$ and every $g \in G_1$.

\item[{\bf A4.}]  For all $g_1, g_2 \in G_1$, we have
                  \[
                    g_1 g_2 \neq g_2 g_1 \quad \mbox{ if and only if } \quad \varphi (g_1) \varphi (g_2) \neq \varphi (g_2) \varphi (g_1) \,,
                  \]
                  and hence also
                  \begin{equation} \label{eq:comm}
                  g_1 g_2 = g_2 g_1 \quad \mbox{ if and only if } \quad \varphi (g_1) \varphi (g_2) = \varphi (g_2) \varphi (g_1) \,.
                  \end{equation}

\item[{\bf A5.}] There are no three elements $g_1, g_2, g_3 \in G_1$ with the following properties:
                 \begin{enumerate}
                 \smallskip
                 \item[(i)] $\varphi (g_1 g_2) = \varphi (g_1) \varphi (g_2)$ and $\varphi (g_1 g_3) = \varphi (g_3) \varphi (g_1)$.
                 \item[(ii)] $g_1 g_2 \neq g_2 g_1$ and $g_1 g_3 \neq g_3 g_1$.
                 \item[(iii)] $g_2 g_3 = g_3 g_2$.
                 \end{enumerate}

\smallskip
\item[{\bf A6.}] If $g_1, g_2, g_3 \in G_1$ are such that $g_1 g_2 \neq g_2 g_1$ with $\varphi (g_1 g_2) = \varphi (g_1) \varphi (g_2)$ and $g_1 g_3 \neq g_3 g_1$ with $\varphi (g_1 g_3) = \varphi (g_3) \varphi (g_1)$, then
                 \[
                   \begin{aligned}
                     \varphi (g_1 g^{-1}_2) = \varphi (g_1) \varphi (g^{-1}_2) \quad & \mbox{ and  } \quad \varphi (g^{-1}_2 g_1) = \varphi (g^{-1}_2) \varphi (g_1) \,, \\
                     \varphi (g_1 g^{-1}_3) = \varphi (g^{-1}_3) \varphi (g_1) \quad & \mbox{ and } \quad \varphi (g^{-1}_3 g_1) = \varphi (g_1) \varphi (g^{-1}_3) \,.
                   \end{aligned}
                 \]

\item[{\bf A7.}] The sets $A = \{ g \in G_1 \colon \varphi (gh) = \varphi (h) \varphi (g) \mbox{ for all } h \in G_1 \}$ and $B = \{ g \in G_1 \colon \varphi (gh) = \varphi (g) \varphi (h) \mbox{ for all } h \in G_1 \}$ are both subgroups of $G_1$.

\smallskip
\item[{\bf A8.}] $\varphi \colon G_1 \to G_2$ is either a group isomorphism or a group anti-isomorphism.
\end{enumerate}
Clearly, if {\bf A8} holds, then $G_1$ and $G_2$ are isomorphic.

\smallskip
\begin{proof}[Proof of {\bf A1.}]
Let $S = g_1 \bdot \ldots \bdot g_{\ell} \in \mathcal F (G_1)$. Since $\varphi \colon \mathcal F (G_1) \to \mathcal F (G_2)$ is a monoid homomorphism, $\varphi (S)$ has the asserted form. Since $\varphi$ stems from an isomorphism from $\mathcal B (G_1) $ to $\mathcal B (G_2)$, we have that $S \in \mathcal B (G_1)$ if and only if $\varphi (S) \in \mathcal B (G_2)$. Thus, if $g \in G_1$, then $g^{[\ord (g)]} \in \mathcal B (G_1)$, whence
\[
\varphi (g)^{[\ord (g)]} = \varphi \big( g^{[\ord (g)]} \big) \in \mathcal B (G_2) \,,
\]
and hence $\ord ( \varphi (g) ) $ divides $\ord (g)$. Repeating the argument with $\varphi^{-1}$, we infer that $\ord (g) = \ord ( \varphi (g))$. In particular, $G_2$ is a torsion group.
 \qedhere ({\bf A1})
\end{proof}

\smallskip
\begin{proof}[Proof of {\bf A2.}]
Let $g \in G_1$. Then $g \bdot g^{-1} \in \mathcal A (G_1)$, whence $\varphi (g) \bdot \varphi (g^{-1}) \in \mathcal A (G_2)$.
Thus, we either have $\varphi (g) \varphi (g^{-1}) = 1_{G_2}$ or $\varphi (g^{-1}) \varphi (g) = 1_{G_2}$. This implies that $\varphi (g^{-1})$ is the inverse of  $\varphi (g)$. \qedhere ({\bf A2})
\end{proof}

\smallskip
\begin{proof}[Proof of {\bf A3.}]
Let $g_1, g_2 \in G_1$.
Since $(g_1g_2) \bdot g^{-1}_1 \bdot g^{-1}_2 \in \mathcal B (G_1)$, it follows that $\varphi (g_1 g_2) \bdot \varphi (g^{-1}_1) \bdot \varphi (g^{-1}_2) \in \mathcal B (G_2)$.
From the product-one equation, we infer by {\bf A2} that
\[
 \varphi (g_1 g_2) = \varphi (g_1) \varphi (g_2) \quad \mbox{ or } \quad \varphi (g_1 g_2) = \varphi (g_2) \varphi (g_1) \,.
\]
In particular, if $g_1 = g_2$, then $\varphi (g^{2}_1) = \varphi (g_1)^{2}$, whence an inductive argument ensures that $\varphi (g^{n}_1) = \varphi (g_1)^{n}$ for all   $n \in \N$.
Moreover, if we replace $g_1 = g_2$ with $g^{-1}_1$, then it follows by {\bf A2} that $\varphi (g^{-2}_1) = \varphi (g_1)^{-2}$, whence an inductive argument ensures again that $\varphi (g^{-n}_1) = \varphi (g_1)^{-n}$ for all  $n \in \N$. \qedhere ({\bf A3})
\end{proof}

\smallskip
\begin{proof}[Proof of {\bf A4.}]
Let $g_1, g_2 \in G_1$.

(i) Suppose that $g_1g_2 \ne g_2g_1$. Since $\varphi$ is bijective, it follows that $\varphi (g_1 g_2) \neq \varphi (g_2 g_1)$.
By {\bf A3}, we have that
\[
\varphi (g_1 g_2) = \varphi (g_1) \varphi (g_2) \quad \text{ or } \quad  \varphi (g_1 g_2) = \varphi (g_2) \varphi (g_1)
\]
and
\[
  \varphi (g_2 g_1) = \varphi (g_1) \varphi (g_2) \quad \text{ or } \quad \varphi (g_2 g_1) = \varphi (g_2) \varphi (g_1) \,.
\]
Since $\varphi (g_1 g_2) \neq \varphi (g_2 g_1)$, we infer that $\varphi (g_1) \varphi (g_2) \neq \varphi (g_2) \varphi (g_1)$.

(ii) Suppose that $\varphi (g_1) \varphi (g_2) \neq \varphi (g_2) \varphi (g_1)$. Applying the inverse map $\varphi^{-1}$, the assertion follows by the same argument.     \qedhere ({\bf A4})
\end{proof}

\smallskip
We will use the following fact, which can be directly derived from {\bf A3} and {\bf A4}, without further mention: If $gh \neq hg$ for $g, h \in G_1$, then $\varphi (gh) = \varphi (g) \varphi (h)$ implies that $\varphi (hg) = \varphi (h) \varphi (g)$, while $\varphi (gh) = \varphi (h) \varphi (g)$ implies that $\varphi (hg) = \varphi (g) \varphi (h)$.

\smallskip
\begin{proof}[Proof of {\bf A5.}]
Assume to the contrary that there are three elements $g_1, g_2, g_3 \in G_1$ satisfying the given properties (i)-(iii).
Then, since $\varphi$ is bijective, (ii) ensures that
\[
  \varphi (g_1 g_2) \neq \varphi (g_2 g_1) \quad \text{ and } \quad \varphi (g_1 g_3) \neq \varphi (g_3 g_1) \,,
\]
whence we infer, by {\bf A3} and (i), that
\begin{equation} \label{eq:anti2}
  \varphi (g_2 g_1) = \varphi (g_2) \varphi (g_1) \quad \mbox{ and } \quad \varphi (g_3 g_1) = \varphi (g_1) \varphi (g_3) \,.
\end{equation}
Now, we claim that, for all $n \in \N$, we have the following identities:
\[
  \begin{aligned}
    (\alpha_1) \quad \hspace{42pt} g_1^ng_2 \ne g_2g_1^n \quad & \mbox{ and } \quad (\alpha_2) \quad g_1^ng_3 \ne g_3g_1^n \,, \\
    (\beta_1) \quad \varphi (g_1^ng_2) = \varphi (g_1^n)\varphi (g_2) \quad & \mbox{ and } \quad (\beta_2) \quad \varphi (g_1^ng_3) = \varphi( g_3) \varphi (g_1^n) \,, \\
    (\gamma_1) \quad \varphi (g_2 g^{n}_1) = \varphi (g_2) \varphi (g^{n}_1) \quad & \mbox{ and } \quad (\gamma_2) \quad \varphi (g_3 g^{n}_1) = \varphi (g^{n}_1) \varphi (g_3) \,.
  \end{aligned}
\]
Clearly, this gives a contradiction for $n=\ord (g_1)$. We proceed by induction on $n$. For $n=1$, the claim holds. Let $n \ge 2$ and suppose the claim holds for all positive integers smaller than $n$.
Let $i, j \in \mathbb N$ with $i + j = n$. Then, by the inductive hypothesis, we obtain the following properties:
\begin{equation} \label{eq:4}
  g^{i}_1 g_2 \neq g_2 g^{i}_1 \quad \mbox{ and } \quad g^{i}_1 g_3 \neq g_3 g^{i}_1 \,,
\end{equation}
\begin{equation} \label{eq:5}
  \varphi \big( g^{i}_1 g_2 \big) = \varphi \big( g^{i}_1 \big) \varphi (g_2) \quad \mbox{ and } \quad \varphi \big( g^{i}_1 g_3 \big) = \varphi (g_3) \varphi \big( g^{i}_1 \big) \,,
\end{equation}
\begin{equation} \label{eq:6}
  \varphi \big( g_2 g^{i}_1 \big) = \varphi (g_2) \varphi \big( g^{i}_1 \big) \quad \mbox{ and } \quad \varphi \big( g_3 g^{i}_1 \big) = \varphi \big( g^{i}_1 \big) \varphi (g_3) \,,
\end{equation}

\noindent
and also

\begin{equation} \label{eq:7}
  g^{j}_1 g_2 \neq g_2 g^{j}_1 \quad \mbox{ and } \quad g^{j}_1 g_3 \neq g_3 g^{j}_1 \,,
\end{equation}
\begin{equation} \label{eq:8}
  \varphi \big( g^{j}_1 g_2 \big) = \varphi \big( g^{j}_1 \big) \varphi (g_2) \quad \mbox{ and } \quad \varphi \big( g^{j}_1 g_3 \big) = \varphi (g_3) \varphi \big( g^{j}_1 \big) \,,
\end{equation}
\begin{equation} \label{eq:9}
  \varphi \big( g_2 g^{j}_1 \big) = \varphi (g_2) \varphi \big( g^{j}_1 \big) \quad \mbox{ and } \quad \varphi \big( g_3 g^{j}_1 \big) = \varphi \big( g^{j}_1 \big) \varphi (g_3) \,.
\end{equation}

\noindent
We will need the following equations, namely that
\begin{equation} \label{eq:case1-4}
  \varphi \big( g^{j}_1 g_3 \big) \varphi \big( g^{i}_1 g_2 \big) \overset{(\ref{eq:8})}{=} \varphi (g_3) \varphi \big( g^{j}_1 \big) \varphi \big( g^{i}_1 g_2 \big) \overset{(\ref{eq:5})}{=} \varphi (g_3) \varphi \big( g^{j}_1 \big) \varphi \big( g^{i}_1 \big) \varphi (g_2) \overset{{\bf A3}}{=} \varphi (g_3) \varphi (g^{n}_1) \varphi (g_2) \,.
\end{equation}

\smallskip
\noindent
$\underline{\text{Proof of} \ (\beta_2).}$
Since $g^{n}_1 g_3 = g^{j}_1 \big( g^{i}_1 g_3 \big)$, {\bf A3} implies  that either
\[
  \varphi (g^{n}_1 g_3) = \varphi \big( g^{j}_1 \big) \varphi \big( g^{i}_1 g_3 \big) \quad \mbox{ or } \quad \varphi (g^{n}_1 g_3) = \varphi \big( g^{i}_1 g_3 \big) \varphi \big( g^{j}_1 \big) \,.
\]
If the second equation holds, then
\[
  \varphi (g^{n}_1 g_3) = \varphi \big( g^{i}_1 g_3 \big) \varphi \big( g^{j}_1 \big) \overset{(\ref{eq:5})}{=} \varphi (g_3) \varphi \big( g^{i}_1 \big) \varphi \big( g^{j}_1 \big) \overset{{\bf A3}}{=} \varphi (g_3) \varphi (g^{n}_1) \,,
\]
whence $(\beta_2)$ holds. Assume to the contrary that the first equation holds, whence
$\varphi (g^{n}_1 g_3) = \varphi \big( g^{j}_1 \big) \varphi \big( g^{i}_1 g_3 \big)$.
Then, we have
\begin{equation} \label{eq:11n3}
  \varphi (g^{n}_1 g_3) = \varphi \big( g^{j}_1 \big) \varphi \big( g^{i}_1 g_3 \big) \overset{(\ref{eq:5})}{=} \varphi \big( g^{j}_1 \big) \varphi (g_3) \varphi \big( g^{i}_1 \big) \overset{(\ref{eq:9})}{=} \varphi \big( g_3 g^{j}_1 \big) \varphi \big( g^{i}_1 \big) \,.
\end{equation}
Hence, $\varphi (g^{n}_1 g_3) \bdot \varphi \big( g_3 g^{j}_1 \big)^{-1} \bdot \varphi \big( g^{i}_1 \big)^{-1} \in \mathcal B (G_2)$, and so we infer by {\bf A1} and {\bf A2} that $g^{n}_1 g_3 \bdot g^{-j}_1 g^{-1}_3 \bdot g^{-i}_1 \in \mathcal B (G_1)$.
Since $g^{j}_1 g_3 \neq g_3 g^{j}_1$, the product-one equation ensures that $g^{n}_1 g_3 = g_3 g^{n}_1$, equivalently $\varphi (g^{n}_1) \varphi (g_3) = \varphi (g_3) \varphi (g^{n}_1) = \varphi (g_1^ng_3)$ by (\ref{eq:comm}) and {\bf A3}. Hence,
\[
  \varphi \big( g^{j}_1 g_3 \big) \varphi \big( g^{i}_1 \big) \overset{(\ref{eq:8})}{=} \varphi (g_3) \varphi \big( g^{j}_1 \big) \varphi \big( g^{i}_1 \big) \overset{{\bf A3}}{=} \varphi (g_3) \varphi (g^{n}_1) = \varphi (g^{n}_1) \varphi (g_3) \overset{{\bf A3}}{=} \varphi (g^{n}_1 g_3) \overset{(\ref{eq:11n3})}{=} \varphi \big( g_3 g^{j}_1 \big) \varphi \big( g^{i}_1 \big) \,,
\]
and thus $\varphi \big( g^{j}_1 g_3 \big) = \varphi \big( g_3 g^{j}_1 \big)$.
Since $\varphi$ is bijective, we have $g^{j}_1 g_3 = g_3 g^{j}_1$, a contradiction to (\ref{eq:7}).

\smallskip
\noindent
$\underline{\text{Proof of} \ (\alpha_1).}$
In view of (\ref{eq:case1-4}), we obtain that
\begin{equation} \label{eq:case1-5}
  \varphi \big( g^{j}_1 g_3 \big) \varphi \big( g^{i}_1 g_2 \big) = \varphi (g_3) \varphi (g^{n}_1) \varphi (g_2) \overset{(\beta_2)}{=} \varphi (g^{n}_1 g_3) \varphi (g_2) \,,
\end{equation}
and again by {\bf A3}, we infer that either
\[
  \varphi \big( g^{j}_1 g_3 \big) \varphi \big( g^{i}_1 g_2 \big) = \varphi ((g^{n}_1 g_3) g_2) \quad \mbox{ or } \quad \varphi \big( g^{j}_1 g_3 \big) \varphi \big( g^{i}_1 g_2 \big) = \varphi (g_2 (g^{n}_1 g_3)) \,.
\]
If $\varphi \big( g^{j}_1 g_3 \big) \varphi \big( g^{i}_1 g_2 \big) = \varphi (g^{n}_1 g_3 g_2)$, then, by {\bf A1} and {\bf A2}, we get a product-one sequence $g^{j}_1 g_3 \bdot g^{i}_1 g_2 \bdot g^{-1}_2 g^{-1}_3 g^{-n}_1 \in \mathcal B (G_1)$.
Since $g_2 g_3 = g_3 g_2$ by (iii), we can see from the product-one equation that either $g^{j}_1 g_2 = g_2 g^{j}_1$ (a contradiction to (\ref{eq:7})) or $g^{i}_1 g_3 = g_3 g^{i}_1$ (a contradiction to (\ref{eq:4})).
Hence, we obtain that
\begin{equation} \label{eq:case1-6}
  \varphi \big( g^{j}_1 g_3 \big) \varphi \big( g^{i}_1 g_2 \big) = \varphi (g_2 g^{n}_1 g_3) \,,
\end{equation}
and by {\bf A3}, we also have that either
\[
  \varphi (g_2 g^{n}_1 g_3) = \varphi \big( (g^{j}_1g_3) (g^{i}_1 g_2) \big) \quad \mbox{ or } \quad \varphi (g_2 g^{n}_1 g_3) = \varphi \big( (g^{i}_1 g_2) (g^{j}_1 g_3) \big) \,.
\]
Since $g^{i}_1 g_2 \neq g_2 g^{i}_1$, we have $\varphi (g_2 g^{n}_1 g_3 ) = \varphi \big( g^{j}_1 g_3 g^{i}_1 g_2 \big)$, and  since $\varphi$ is bijective, $g_2 g^{n}_1 g_3 = g^{j}_1 g_3 g^{i}_1 g_2$, i.e., $g^{-n}_1 g_2 g^{n}_1 g^{-1}_2 = g^{-i}_1 g_3 g^{i}_1 g^{-1}_3$ (since $g_2 g_3 = g_3 g_2$).
Since $g^{i}_1 g_3 \neq g_3 g^{i}_1$, we infer that
\begin{equation} \label{eq:odd-g2}
  g^{n}_1 g_2 \neq g_2 g^{n}_1 \,,
\end{equation}
and so $(\alpha_1)$ holds.

\smallskip
\noindent
$\underline{\text{Proof of} \ (\beta_1).}$
Assume to the contrary, that $\varphi (g^{n}_1 g_2) \ne \varphi (g^{n}_1) \varphi (g_2)$. Then we have
\[
  \varphi (g^{n}_1 g_2) \overset{{\bf A3}}{=} \varphi (g_2) \varphi (g^{n}_1) \overset{{\bf A3}}{=} \varphi (g_2) \varphi \big( g^{j}_1 \big) \varphi \big( g^{i}_1 \big) \overset{(\ref{eq:9})}{=} \varphi \big( g_2 g^{j}_1 \big) \varphi \big( g^{i}_1 \big) \,,
\]
and by {\bf A1} and {\bf A2}, we obtain that $g^{n}_1 g_2 \bdot g^{-j}_1 g^{-1}_2 \bdot g^{-i}_1 \in \mathcal B (G_1)$.
From the product-one equation, we can see that either $g^{j}_1 g_2 = g_2 g^{j}_1$ (a contradiction to (\ref{eq:7})) or $g^{n}_1 g_2 = g_2 g^{n}_1$ (a contradiction to (\ref{eq:odd-g2})), whence $\varphi (g^{n}_1 g_2) = \varphi (g^{n}_1) \varphi (g_2)$, and so $(\beta_1)$ holds.

\smallskip
\noindent
$\underline{\text{Proof of} \ (\gamma_1).}$
In view of ($\alpha_1$) and ($\beta_1$), we infer by  {\bf A3} that $\varphi (g_2 g^{n}_1) = \varphi (g_2) \varphi (g^{n}_1)$, and so $(\gamma_1)$ holds.

\smallskip
\noindent
$\underline{\text{Proof of} \ (\alpha_2).}$
We have
\[
  \varphi (g_2 g^{n}_1 g_3) \overset{(\ref{eq:case1-6})}{=} \varphi \big( g^{j}_1 g_3 \big) \varphi \big( g^{i}_1 g_2 \big) \overset{(\ref{eq:case1-4})}{=} \varphi (g_3) \varphi (g^{n}_1) \varphi (g_2) \overset{(\beta_1)}{=} \varphi (g_3) \varphi (g^{n}_1 g_2) \,.
\]
If $\varphi (g_3) \varphi (g^{n}_1 g_2) = \varphi ((g^{n}_1 g_2) g_3)$, then $\varphi (g_2 g^{n}_1 g_3) = \varphi (g^{n}_1 g_2 g_3)$, and since $\varphi$ is bijective, $g_2 g^{n}_1 = g^{n}_1 g_2$, a contradiction to (\ref{eq:odd-g2}).
Thus, we infer by {\bf A3} that
\[
  \varphi \big( g^{j}_1 g_3 \big) \varphi \big( g^{i}_1 g_2 \big) = \varphi (g_3) \varphi (g^{n}_1 g_2) = \varphi (g_3 (g^{n}_1 g_2)) \,,
\]
and again by {\bf A3}, we also have that either
\[
  \varphi (g_3 g^{n}_1 g_2) = \varphi \big( (g^{j}_1 g_3) (g^{i}_1 g_2) \big) \quad \mbox{ or } \quad \varphi (g_3 g^{n}_1 g_2) = \varphi \big( (g^{i}_1 g_2) (g^{j}_1 g_3) \big) \,.
\]
Since $g^{j}_1 g_3 \neq g_3 g^{j}_1$, we have $\varphi (g_3 g^{n}_1 g_2) = \varphi \big( g^{i}_1 g_2 g^{j}_1 g_3)$, and since $\varphi$ is bijective, $g_3 g^{n}_1 g_2 = g^{i}_1 g_2 g^{j}_1 g_3$, i.e., $g^{-j}_1 g_2 g^{j}_1 g^{-1}_2 = g^{-n}_1 g_3 g^{n}_1 g^{-1}_3$ (since $g_2 g_3 = g_3 g_2$).
Since $g^{j}_1 g_2 \neq g_2 g^{j}_1$, we  infer that
\begin{equation} \label{eq:odd-g3}
  g^{n}_1 g_3 \neq g_3 g^{n}_1 \,,
\end{equation}
and so $(\alpha_2)$ holds.

\smallskip
\noindent
$\underline{\text{Proof of} \ (\gamma_2).}$ In view of ($\alpha_2$) and ($\beta_2$), we infer by {\bf A3} that $\varphi (g_3 g^{n}_1) = \varphi (g^{n}_1) \varphi (g_3)$, and so $(\gamma_2$) holds.    \qedhere ({\bf A5})
\end{proof}

\smallskip
\begin{proof}[Proof of {\bf A6.}]
Let $g_1, g_2, g_3 \in G_1$ be such that
\begin{equation} \label{eq:noncomm}
  g_1 g_2 \neq g_2 g_1 \quad \mbox{ and } \quad g_1 g_3 \neq g_3 g_1 \,,
\end{equation}
\begin{equation} \label{eq:var-noncomm}
  \varphi (g_1 g_2) = \varphi (g_1) \varphi (g_2) \quad \mbox{ and } \quad \varphi (g_1 g_3) = \varphi (g_3) \varphi (g_1) \,.
\end{equation}

\smallskip
(i) First we show that $\varphi (g_1 g^{-1}_2) = \varphi (g_1) \varphi (g^{-1}_2)$.
By {\bf A3}, we assume to the contrary that $\varphi (g_1 g^{-1}_2) = \varphi (g^{-1}_2) \varphi (g_1)$. This implies that
\[
  \varphi (g_1 g_2) \varphi (g_1 g^{-1}_2) \overset{(\ref{eq:var-noncomm})}{=} \varphi (g_1) \varphi (g_2) \varphi (g^{-1}_2) \varphi (g_1) \overset{{\bf A2}}{=} \varphi (g_1) \varphi (g_1) \overset{{\bf A3}}{=} \varphi (g^{2}_1) \,.
\]
Thus, by {\bf A1} and {\bf A2}, we get a product-one sequence $g_1 g_2 \bdot g_1 g^{-1}_2 \bdot g^{-2}_1 \in \mathcal B (G_1)$.
From the product-one equation, we infer that $g_1 g_2 = g_2 g_1$, a contradiction to (\ref{eq:noncomm}).

\smallskip
(ii) Next, we show that  $\varphi (g^{-1}_2 g_1) = \varphi (g^{-1}_2) \varphi (g_1)$. Since $\varphi$ is bijective,  (\ref{eq:noncomm}) implies that $\varphi (g^{-1}_2 g_1) \ne \varphi (g_1 g^{-1}_2)$ . Hence,  by (i) and {\bf A3}, we infer that $\varphi (g^{-1}_2 g_1) = \varphi (g^{-1}_2) \varphi (g_1)$.

\smallskip
(iii) We show that $\varphi (g_1 g^{-1}_3) = \varphi (g^{-1}_3) \varphi (g_1)$. {\bf A3} implies that either $\varphi (g_1 g^{-1}_3) = \varphi (g_1) \varphi (g^{-1}_3)$ or $\varphi (g_1 g^{-1}_3) = \varphi (g^{-1}_3) \varphi (g_1)$.
Assume to the contrary, that $\varphi (g_1 g^{-1}_3) = \varphi (g_1) \varphi (g^{-1}_3)$. Then
\[
  \varphi (g_1 g^{-1}_3) \varphi (g_1 g_3) \overset{(\ref{eq:var-noncomm})}{=} \varphi (g_1) \varphi (g^{-1}_3) \varphi (g_3) \varphi (g_1) \overset{{\bf A2}}{=} \varphi (g_1) \varphi (g_1) \overset{{\bf A3}}{=} \varphi (g^{2}_1) \,.
\]
Thus, by {\bf A1} and {\bf A2}, we get a product-one sequence $g_1 g^{-1}_3 \bdot g_1 g_3 \bdot g^{-2}_1 \in \mathcal B (G_1)$.
From the product-one equation, we obtain that $g_1 g_3 = g_3 g_1$, a contradiction to (\ref{eq:noncomm}).

\smallskip
(iv) We show that $\varphi ( g^{-1}_3 g_1) =  \varphi (g_1) \varphi (g^{-1}_3)$.
Since $\varphi$ is bijective,  (\ref{eq:noncomm}) implies that $ \varphi (g^{-1}_3 g_1) \ne \varphi (g_1 g^{-1}_3)$. Thus,   by (iii) and {\bf A3}, we infer  that $\varphi (g^{-1}_3 g_1) = \varphi (g_1) \varphi (g^{-1}_3)$.  \qedhere ({\bf A6})
\end{proof}

\smallskip
\begin{proof}[Proof of {\bf A7.}]
Since $\varphi (1_{G_1}) = 1_{G_2}$ by {\bf A1}, it follows that $\varphi (1_{G_1} x) = \varphi (x) = \varphi (x) 1_{G_2} = \varphi (x) \varphi (1_{G_1})$ for all $x \in G_1$, and thus $1_{G_1} \in A$.
For $g_1, g_2 \in A$ and $x \in G_1$,
\[
  \varphi ((g_1g_2)x) = \varphi (g_2 x) \varphi (g_1) = \varphi (x) \varphi (g_2) \varphi (g_1) = \varphi (x) \varphi (g_1g_2) \,,
\]
and hence $g_1g_2 \in A$.
Now let $g \in A$ and $x \in G_1$.
If $gx = xg$, then
\[
  \begin{aligned}
    \varphi (g^{-1} x) & = \varphi ((x^{-1}g)^{-1}) \overset{{\bf A2}}{=} \varphi (x^{-1}g)^{-1} = \varphi (gx^{-1})^{-1} = \big( \varphi (x^{-1})\varphi (g) \big)^{-1}  \\
                       & \overset{(\ref{eq:comm})}{=} \big( \varphi (g) \varphi (x^{-1}) \big)^{-1} = \varphi (x^{-1})^{-1} \varphi (g)^{-1} \overset{{\bf A2}}{=} \varphi (x) \varphi (g^{-1}) \,.
  \end{aligned}
\]
If $gx \neq xg$, then
\[
  \varphi (g^{-1} x) = \varphi ((x^{-1}g)^{-1}) \overset{{\bf A2}}{=} \varphi (x^{-1}g)^{-1} \overunderset{(\ref{eq:comm})}{{\bf A3}}{=} \big( \varphi (g) \varphi (x^{-1}) \big)^{-1} = \varphi (x^{-1})^{-1} \varphi (g)^{-1} \overset{{\bf A2}}{=} \varphi (x) \varphi (g^{-1}) \,.
\]
In either case, we infer that $g^{-1} \in A$, and hence $A$ is a subgroup of $G_1$.
The same argument shows that $B$ is also a subgroup of $G_1$.    \qedhere ({\bf A7})
\end{proof}

\smallskip
\begin{proof}[Proof of {\bf A8.}]
Let $A$ and $B$ be subgroups of $G_1$ as defined in {\bf A7}.
We have to show that, for all $g \in G_1$, either $g \in A$ or $g \in B$ holds.
This implies that $G_1 = A \cup B$.
Since a group is not the union of two proper subgroups, we obtain that either $G_1 = B$ and $\varphi$ is an isomorphism, or that $G_1 = A$ and $\varphi$ is an anti-homomorphism, whence we are done.
Now, assume to the contrary that there exists $g_1 \in G_1$ such that $g_1$ is neither in $A$ nor in $B$.
Then, there are $g_2, g_3 \in G_1$ such that
\begin{equation} \label{eq:anti0}
  \varphi (g_1g_2) \ne \varphi (g_2)\varphi (g_1) \quad \text{ and } \quad \varphi (g_1g_3) \ne \varphi (g_1)\varphi (g_3) \,,
\end{equation}
and hence {\bf A3} implies that
\begin{equation} \label{eq:anti}
 \varphi (g_1 g_2) = \varphi (g_1) \varphi (g_2) \quad \mbox{ and } \quad \varphi (g_1 g_3) = \varphi (g_3) \varphi (g_1) \,.
\end{equation}
We assert that either
\[
g_1 g_2 = g_2 g_1 \quad \text{ or} \quad  g_1 g_3 = g_3 g_1 \,.
\]
If this holds, then,
in view of (\ref{eq:comm}), we obtain  that $\varphi (g_1)$ commutes  either with $\varphi (g_2)$ or with $\varphi (g_3)$, a contradiction. So, let us assume to the contrary that
\begin{equation} \label{eq:assumption}
  g_1 g_2 \neq g_2 g_1 \quad \text{ and } \quad g_1 g_3 \neq g_3 g_1 \,.
\end{equation}
In view of (\ref{eq:anti}) and (\ref{eq:assumption}), {\bf A5} ensures that
\begin{equation} \label{eq:g_2g_3}
  g_2 g_3 \neq g_3 g_2 \,.
\end{equation}
Moreover, since $\varphi$ is bijective, {\bf A3} and (\ref{eq:anti}) imply
\begin{equation} \label{eq:anti3}
  \varphi (g_2 g_1) = \varphi (g_2) \varphi (g_1) \quad \mbox{ and } \quad \varphi (g_3 g_1) = \varphi (g_1) \varphi (g_3) \,.
\end{equation}

\noindent
In order to avoid some case redundancies, we consider the map $\varphi^{\op} \colon G_1 \to G^{\op}_2$, defined by $\varphi^{\op} (g) = \varphi (g)$ for all $g \in G_1$ (see the discussion at the end of Section \ref{2}).
Obviously, the assertions {\bf A1}-{\bf A4} hold true for $\varphi^{\op}$.
Since $\varphi^{\op} (g_1g_2) = \varphi^{\op} (g_2) \cdot^{\op} \varphi^{\op} (g_1)$ and $\varphi^{\op} (g_1 g_3) = \varphi^{\op} (g_1) \cdot^{\op} \varphi^{\op} (g_3)$, it is easy to see that {\bf A5} and {\bf A6} also hold true for $\varphi^{\op}$ by swapping the role between $g_2$ and $g_3$.

By {\bf A3}, we obtain that $\varphi (g_2 g_3) \in \{ \varphi (g_2) \varphi (g_3), \, \varphi (g_3) \varphi (g_2) \}$.
If $\varphi (g_2 g_3) = \varphi (g_3) \varphi (g_2)$, then $\varphi^{\op} (g_2 g_3) = \varphi^{\op} (g_2) \cdot^{\op} \varphi^{\op} (g_3)$, and thus by replacing $\varphi$ by $\varphi^{\op}$ if necessary, we may assume without further restriction that
\begin{equation} \label{eq:var_g_2g_3}
  \varphi (g_2 g_3) = \varphi (g_2) \varphi (g_3) \quad \mbox{ and } \quad \varphi (g_3 g_2) = \varphi (g_3) \varphi (g_2) \,,
\end{equation}
where the second equality follows from the first and {\bf A3}.
Then, we have
\[
  \varphi (g_1 g^{-1}_2) \varphi (g_2 g_3) \overset{{\bf A6}}{=} \varphi (g_1) \varphi (g^{-1}_2) \varphi (g_2) \varphi (g_3) \overset{{\bf A2}}{=} \varphi (g_1) \varphi (g_3) \overset{(\ref{eq:anti3})}{=} \varphi (g_3 g_1) \,,
\]
whence we infer by {\bf A1} and {\bf A2} that $g_1 g^{-1}_2 \bdot g_2 g_3 \bdot g^{-1}_1 g^{-1}_3 \in \mathcal B (G_1)$.
Since $g_1 g_3 \neq g_3 g_1$, the product-one equation shows that $g^{-1}_3 g_2 g_3 = g_1 g_2 g^{-1}_1$.
Furthermore,
\[
  \varphi (g_2 g_3) \varphi (g_1 g^{-1}_3) \overset{{\bf A6}}{=} \varphi (g_2) \varphi (g_3) \varphi (g^{-1}_3) \varphi (g_1) \overset{{\bf A2}}{=} \varphi (g_2) \varphi (g_1) \overset{(\ref{eq:anti3})}{=} \varphi (g_2 g_1) \,,
\]
whence we infer, again by {\bf A1} and {\bf A2}, that $g_2 g_3 \bdot g_1 g^{-1}_3 \bdot g^{-1}_1 g^{-1}_2 \in \mathcal B (G_1)$.
Since $g_1 g_3 \neq g_3 g_1$, the product-one equation shows that $g^{-1}_3 g_2 g_3 = g^{-1}_1 g_2 g_1$.
Therefore, we obtain that
\begin{equation} \label{eq:g2-comm}
  g_1 g_2 g^{-1}_1 = g^{-1}_3 g_2 g_3 = g^{-1}_1 g_2 g_1 \,, \quad \mbox{ whence } \,\, g^{2}_1 g_2 = g_2 g^{2}_1 \,.
\end{equation}
We distinguish two cases.

\smallskip
\noindent
CASE 1: \ $g_3 g^{2}_1 \neq g^{2}_1 g_3$.

By {\bf A3}, we have either $\varphi (g_3 g^{2}_1) = \varphi (g_3) \varphi (g^{2}_1)$ or $\varphi (g_3 g^{2}_1) = \varphi (g^{2}_1) \varphi (g_3)$.
If $\varphi (g_3 g^{2}_1) = \varphi (g_3) \varphi (g^{2}_1)$, then
\[
  \varphi (g_3 g^{2}_1) \overset{{\bf A3}}{=} \varphi (g_3) \varphi (g_1) \varphi (g_1) \overset{(\ref{eq:anti})}{=} \varphi (g_1 g_3) \varphi (g_1) \,,
\]
whence, by {\bf A1} and {\bf A2}, we get a product-one sequence $g_3 g^{2}_1 \bdot g^{-1}_3 g^{-1}_1 \bdot g^{-1}_1 \in \mathcal B (G_1)$.
From the product-one equation, we can see that either $g_1 g_3 = g_3 g_1$ (a contradiction to (\ref{eq:assumption})) or $g^{2}_1 g_3 = g_3 g^{2}_1$ (a contradiction to the hypothesis of CASE 1).
Thus, we consequently have the following conditions:
\[
  \begin{aligned}
    g_3 g_2 \neq g_2 g_3 & \quad \mbox{ and } \quad g_3 g^{2}_1 \neq g^{2}_1 g_3 \,, \\
    \varphi (g_3 g_2) = \varphi (g_3) \varphi (g_2) & \quad \mbox{ and } \quad \varphi (g_3 g^{2}_1) = \varphi (g^{2}_1) \varphi (g_3) \,.
  \end{aligned}
\]
In view of (\ref{eq:g2-comm}), we obtain a  triple $(g_3, g_2, g^{2}_1)$ satisfying the three conditions  in {\bf A5},  a contradiction.

\smallskip
\noindent
CASE 2: \ $ g_3 g^{2}_1 = g^{2}_1 g_3$.

Then, we have
\[
  \varphi (g_1 g_3) \varphi (g_1 g_2) \overset{(\ref{eq:anti})}{=} \varphi (g_3) \varphi (g_1) \varphi (g_1) \varphi (g_2) \overset{{\bf A3}}{=} \varphi (g_3) \varphi (g^{2}_1) \varphi (g_2) \overunderset{(\ref{eq:comm})}{{\bf A3}}{=} \varphi (g_3 g^{2}_1) \varphi (g_2) \,,
\]
and it follows by {\bf A3} that
\[
  \varphi (g_1 g_3) \varphi (g_1 g_2) = \varphi ((g_3 g^{2}_1) g_2) \quad \mbox{ or } \quad \varphi (g_1 g_3) \varphi (g_1 g_2) = \varphi (g_2 (g_3 g^{2}_1)) \,.
\]
If the latter case holds true, then since $g^{2}_1g_2 = g_2 g^{2}_1$ and $g^{2}_1 g_3 = g_3 g^{2}_1$, we obtain that
\[
  \varphi (g_2 g_3) \varphi (g^{2}_1) \overunderset{(\ref{eq:comm})}{{\bf A3}}{=} \varphi ((g_2 g_3)g^{2}_1) = \varphi (g_1 g_3)\varphi (g_1 g_2) \overunderset{(\ref{eq:anti})}{{\bf A3}}{=} \varphi (g_3) \varphi(g^{2}_1) \varphi (g_2) \overset{(\ref{eq:comm})}{=} \varphi (g_3) \varphi (g_2) \varphi (g^{2}_1) \,.
\]
Then, $\varphi (g_2) \varphi (g_3) \overset{(\ref{eq:var_g_2g_3})}{=} \varphi (g_2 g_3) = \varphi (g_3) \varphi (g_2)$, and it follows by (\ref{eq:comm}) that $g_2 g_3 = g_3 g_2$, a contradiction to (\ref{eq:g_2g_3}).
Thus, we must have that $\varphi (g_1 g_3) \varphi (g_1 g_2) = \varphi (g_3 g^{2}_1 g_2)$.
On the other hand, by {\bf A3}, we have either $\varphi (g_1g_2g_1g_3) = \varphi (g_1g_2) \varphi (g_1g_3)$ or $\varphi (g_1g_2g_1g_3) = \varphi (g_1g_3) \varphi (g_1g_2)$.
If the former case holds true, then
\[
  \varphi (g_1g_2g_1g_3) = \varphi (g_1g_2) \varphi (g_1g_3) \overset{(\ref{eq:anti})}{=} \varphi (g_1) \varphi (g_2) \varphi (g_3) \varphi (g_1) \overset{(\ref{eq:var_g_2g_3})}{=} \varphi (g_1) \varphi (g_2g_3) \varphi (g_1) \,,
\]
whence, by {\bf A1} and {\bf A2}, we get a product-one sequence $g_1g_2g_1g_3 \bdot g^{-1}_1 \bdot g^{-1}_3 g^{-1}_2 \bdot g^{-1}_1 \in \mathcal B (G_1)$.
Since $g_3 g^{2}_1 = g^{2}_1 g_3$, we can see from the product-one equation that either $g_1g_3 = g_3g_1$ or $g_1g_2 = g_2g_1$, a contradiction to (\ref{eq:assumption}).
Thus, we have $\varphi (g_1 g_2 g_1 g_3) = \varphi (g_1g_3) \varphi (g_1g_2) = \varphi (g_3 g^{2}_1 g_2)$, and since $\varphi$ is bijective, $g_1 g_2 g_1 g_3 = g_3 g^{2}_1 g_2 = g^{2}_1 g_3 g_2$, and so $g_2 (g_1 g_3) = (g_1 g_3) g_2$.
Note that $g_1 g_3 \neq g_2$, because otherwise
\[
  \begin{aligned}
    \varphi (g_1) \varphi (g_2) & \overset{(\ref{eq:anti})}{=} \varphi (g_1 g_2) = \varphi (g_1 (g_1 g_3)) = \varphi (g^{2}_1 g_3) \\
                                & \overunderset{(\ref{eq:comm})}{{\bf A3}}{=} \varphi (g_3) \varphi (g^{2}_1) \overset{{\bf A3}}{=} \varphi (g_3) \varphi (g_1) \varphi (g_1) \overset{(\ref{eq:anti})}{=} \varphi (g_1 g_3) \varphi (g_1) = \varphi (g_2) \varphi (g_1)
  \end{aligned}
\]
implies,  by (\ref{eq:comm}), that $g_1 g_2 = g_2 g_1$, a contradiction to (\ref{eq:assumption}).
Moreover, $\varphi (g_1 (g_1g_3)) = \varphi (g_1g_3) \varphi (g_1)$, for otherwise
\[
  \varphi (g_1) \varphi (g_1g_3) \overset{{\bf A3}}{=} \varphi (g_1 (g_1g_3)) = \varphi (g^{2}_1 g_3) \overunderset{(\ref{eq:comm})}{{\bf A3}}{=} \varphi (g_1)^{2} \varphi (g_3) \,,
\]
implying that $\varphi (g_1g_3) = \varphi (g_1) \varphi (g_3)$, a contradiction to (\ref{eq:anti0}). Thus,
\[
  \begin{aligned}
    g_1 g_2 \neq g_2 g_1 & \quad \mbox{ and } \quad g_1 (g_1 g_3) \neq (g_1 g_3) g_1 \,, \\
    \varphi (g_1 g_2) = \varphi (g_1) \varphi (g_2) & \quad \mbox{ and} \quad \varphi (g_1 (g_1 g_3)) = \varphi (g_1 g_3) \varphi (g_1) \,,
  \end{aligned}
\]
and since $g_2 (g_1 g_3) = (g_1 g_3) g_2$, we obtain a  triple $(g_1, g_2, g_1g_3)$ satisfying the three conditions in {\bf A5}, a contradiction.   \qedhere ({\bf A8})
\end{proof} \qedhere \end{proof}

\smallskip
\section{Proof of Corollary~\ref{1.2}} \label{4}
\smallskip

For a group $G$, let $\mathsf T (G)$ denote the set of torsion elements. It is well-known that $\mathsf T (G)$ is actually a subgroup of $G$ if $G$ is nilpotent, or if the commutator subgroup $G'$ consists of torsion elements only. Moreover, for a large class of groups, including all solvable groups for example, the commutator subgroup is a torsion group provided that all commutators have finite order (see \cite{Br90}).

\smallskip
\begin{proof}[Proof of Corollary \ref{1.2}]
We may suppose that $G_1$ and $G_2$ are not torsion groups, otherwise the claim follows from Theorem~\ref{1.1}. In particular, each group has infinitely many elements.
Since the commutator subgroups $G_1'$ and $G_2'$ are torsion,  $\mathsf T (G_1)$ and $\mathsf T (G_2)$ are groups.
By Theorem~\ref{1.1}, it suffices to prove that the monoids $\mathcal B (\mathsf T (G_1))$ and $\mathcal B ( \mathsf T (G_2))$ are isomorphic.

Let $i \in [1,2]$. Since the commutator subgroup $G_i'$ is torsion,  \cite[Lemma 3.4]{Fa-Zh22a} implies that $\widehat{\mathcal B (G_i)}$ is Krull and the inclusion
$\widehat{\mathcal B (G_i)} \hookrightarrow \mathcal F (G_i)$ is a divisor theory. The isomorphism from $\mathcal B (G_1)$ to $\mathcal B (G_2)$  induces isomorphisms between their quotient groups and their complete integral closures. If $\Phi \colon \widehat{\mathcal B (G_1)} \to \widehat{\mathcal B (G_2)}$ is this isomorphism, then the uniqueness theorem of divisor theories implies the existence of an isomorphism $\varphi \colon \mathcal F (G_1) \to \mathcal F (G_2)$ with $\varphi \mid_{\widehat{\mathcal B(G_1)}} = \Phi$. In particular, $\varphi \mid_{G_1} \colon G_1 \to G_2$ is a bijection.

Arguing as in {\bf A1} in the proof of Theorem~\ref{1.1}, we infer that $\ord (g) = \ord \big( \varphi (g) \big)$ for every $g \in \mathsf T (G_1)$.
Thus, we obtain that $\varphi \big( \mathsf T (G_1) \big) \subset  \mathsf T (G_2)$ and, by symmetry, we infer that $\varphi \big( \mathsf T (G_1) \big) =  \mathsf T (G_2)$.  This implies that $\varphi \big( \mathcal B ( \mathsf T (G_1) ) \big) = \mathcal B \big( \mathsf T (G_2) \big)$, whence $\mathcal B \big( \mathsf T (G_1) \big)$ and $\mathcal B \big( \mathsf T (G_2) \big)$ are isomorphic.
\end{proof}

\medskip
\noindent
{\bf Acknowledgement.} We would like to thank the reviewer for their time and effort. They provided a detailed list of comments which helped to improve the presentation of our manuscript.

\smallskip
\providecommand{\bysame}{\leavevmode\hbox to3em{\hrulefill}\thinspace}
\providecommand{\MR}{\relax\ifhmode\unskip\space\fi MR }
\providecommand{\MRhref}[2]{%
  \href{http://www.ams.org/mathscinet-getitem?mr=#1}{#2}
}
\providecommand{\href}[2]{#2}

\bigskip

\end{document}